\documentclass[preprint,12pt]{elsarticle}

\usepackage{amsfonts,amsmath,amssymb,amsthm,amstext,txfonts}
\usepackage{extarrows,longtable,graphicx,hyperref,mathrsfs,url}

\newtheorem{theorem}{Theorem}[section]
\newtheorem{proposition}[theorem]{Proposition}
\newtheorem{lemma}[theorem]{Lemma}
\newtheorem{remark}[theorem]{Remark}

\numberwithin{equation}{section}
\allowdisplaybreaks

\begin{document}

\begin{frontmatter}
\title{A tree-like fractal Dirichlet space lying between strong and weak elliptic Harnack inequalities}
\author{Caoxu Huang\fnref{H}}
\ead[H]{1336397107@qq.com}
\address[H]{School of Mathematics, Tianjin University, Tianjin, 300350, China}
\author{Guanhua Liu\fnref{L}}
\ead[L]{liu\_gh@tju.edu.cn}
\address[L]{Center for Applied Mathematics and KL-AAGDM, Tianjin University, Tianjin, 300072, China}

\begin{abstract}
In this paper we construct a self-similar fractal configured as an infinitely branched tree and equip it with a regular self-similar Dirichlet form. Fixing a typical metric ball, we compute accurately the effective resistance between an inner point and the boundary of this ball. Based on this result, anomalous behaviour of the mean exit time from such balls is shown -- the ratio between whose supremum and infimum can be arbitrarily large. Under properly selected self-similar measure, we further show the weak elliptic Harnack inequality holds but the strong analogue fails.
\end{abstract}

\begin{keyword}
self-similar fractal \sep Dirichlet form \sep mean exit time \sep weak elliptic Harnack inequality \sep strong elliptic Harnack inequality

\MSC[2020] 28A80 \sep 31C25 \sep 47D07 \sep 60J46
\end{keyword}
\end{frontmatter}

\section{Introduction}

Harnack inequalities constitute a key part in the DeGiorgi--Nash--Moser theory on elliptic and parabolic partial differential equations, controlling the ratio between the supremum and infimum of a harmonic or caloric function in a metric ball by a constant independent of the function and the ball. Due to the same reason, they also play a significant role in the theory of Markovian processes and fratcal geometry.

Essentially harmonicity or caloricity only concerns the energy form defining weak solutions of involved equations. This way, a parallel theory can be founded on a general Dirichlet space (that is, a metric measure space equipped with a regular Dirichlet form), with necessary modifications decided by the metric measure structure. In particular, the parabolic Harnack inequality (for short, PHI), bridging geometry and energy structure to heat kernel estimates and regularity of solutions, is characterized thoroughly in \cite{bbk,bgk,ghl15} for local, and \cite{ckw-e,ckw-p} for non-local, Dirichlet spaces. A counterpart for the weak parabolic Harnack inequality (for short, wPH) is completed by the second author in \cite{L1} for both local and non-local cases.

PHI and wPH can respectively yield EHI and wEH, that is, strong and weak elliptic Harnack inequalities, but the converse needs further conditions. To be exact, by \cite[Corollary 2.5]{L1} (Corollary 1.2 in the second ArXiv preprint), 
\begin{equation}\label{harnack}
	(\mathrm{PHI})\Leftrightarrow(\mathrm{wPH})\Leftrightarrow(\mathrm{EHI})+(\mathrm{E})\Leftrightarrow(\mathrm{wEH})+(\mathrm{E})
\end{equation}
for a strongly local Dirichlet form on two-sided doubling spaces. Here $(\mathrm{E})$ stands for a two-sided estimate on the mean exit time from metric balls. Till now very little is known on elliptic Harnack inequalities -- there is no characterization on $(\mathrm{EHI})$ or $(\mathrm{wEH})$, and we hardly know anything about the exceptions
\begin{equation}\label{ecp}
	(\mathrm{EHI})\setminus(\mathrm{E})\quad\mbox{and}\quad(\mathrm{wEH})\setminus(\mathrm{EHI}),
\end{equation}
each contradicting $(\mathrm{PHI})$, $(\mathrm{wPH})$ and $(\mathrm{E})$ according to (\ref{harnack}).

These exceptions are completely excluded on classical spaces, since condition $(\mathrm{E})$ is easily established under the time-space scale $t=r^2$ there. Similarly, on post-critically finite fractals, condition $(\mathrm{E})$ is established under scale $t=r^\beta$, with a walk dimension $\beta\ge 2$ (see \cite{hk,L0}), excluding the exceptions again. Meanwhile, on more complicated fractals and other metric spaces, it is very hard to tell whether $(\mathrm{E})$ holds, and hence it is an important issue to distinguish the two exceptional cases in (\ref{ecp}) from regular ones.

Seemingly some works contribute: Barlow and Murugan \cite{bm} show that $(\mathrm{EHI})$ is stable under quasi-isometries; Hu and Yu \cite{hy} show how $(\mathrm{wEH})$ is implied from some isoparametric inequalities. However, treatments in these papers implicitly add $(\mathrm{E})$ from the conditions involved therein, thus irrelevant to the exceptions.

The separation between $(\mathrm{EHI})$ and $(\mathrm{PHI})$ is established in \cite{del} on graphs and \cite{km} on continuous metric measure spaces. The latter realizes our first exceptional case $(\mathrm{EHI})\setminus(\mathrm{E})$, according to (\ref{harnack}). The phenomenon that elliptic regularity holds but parabolic regularity fails is explained accurately using the conformal walk dimension.

Up to now there is still zero knowledge on the second exception $(\mathrm{wEH})\setminus(\mathrm{EHI})$. This paper provides a first step into this field. We construct for the first time a concrete connected example where $(\mathrm{E})$ truly fails under any scale, and show on this example that $(\mathrm{EHI})$ fails but $(\mathrm{wEH})$ holds under a properly selected measure.

Fix an arbitrary compact Polish space $K$ equipped with a separable metrizable topology $\mathcal{T}(K)$ and a Radon measure $\mu$ of full support on $K$.

Recall that (cf.\ \cite{fot,ki}) a symmetric non-negative semi-definite bilinear form $\mathcal{E}$ defined on some linear subspace $\mathcal{F}$ of $L^2(K,\mu)$ is called a closed form, if $\mathcal{F}$, endowed with the inner product $\mathcal{E}_\mathcal{F}(u,w):=\mathcal{E}(u,w)+(u,w)_{L^2(\mu)}$, is a Hilbert space. We denote the corresponding quadratic form $\mathcal{E}(u):=\mathcal{E}(u,u)$, and let $\mathcal{E}(u):=\infty$ throughout $u\in L^2\setminus\mathcal{F}$. A closed form $(\mathcal{E},\mathcal{F})$ is called a \emph{Dirichlet form}, if $\mathcal{F}$ is dense in $L^2(K,\mu)$ and $\mathcal{E}$ satisfies the Markovian property
$$\mathcal{E}((u\vee0)\wedge1)\le\mathcal{E}(u)\quad\mbox{for all}\quad u\in\mathcal{F}.$$
It is called \emph{regular}, if $C(K)\cap\mathcal{F}$ is dense in both $C(K)$ and $\mathcal{F}$. In this case, define the \emph{effective resistance} corresponding to $(\mathcal{E},\mathcal{F})$ by
$$R(A,B)=1/\inf\left\{\mathcal{E}(u):\ u|_A=0,\ u|_B=1\right\}$$
for any two disjoint Borel sets $A,B\subset K$, and write $R(\{x\},\{y\})$, $R(\{x\},A)$ respectively for short as $R(x,y)$, $R(x,A)$. Note if $R(x,y)<\infty$ for all $x,y\in K$, then $R$ becomes a metric generating the original topology $\mathcal{T}(K)$. This metric is frequently used in heat kernel estimates, see \cite{hk,L0} for example.

For an arbitrary open set $\Omega\subset K$, we say $u\in\mathcal{F}$ is \emph{harmonic} in $\Omega$, if $\mathcal{E}(u,\varphi)=0$ for all $\varphi\in\mathcal{F}(\Omega)$. Here $\mathcal{F}(\Omega)$ is the $\mathcal{E}_\mathcal{F}$-completion of the collection of functions in $\mathcal{F}$ supported in $\Omega$. Further, let $G_\mu^\Omega$ be the Green operator of this Dirichlet space so that $G_\mu^\Omega f\in\mathcal{F}(\Omega)$ for every $f\in L^2(\Omega)$, and that for each $g\in\mathcal{F}(\Omega)$,
$$\mathcal{E}\left(G_\mu^\Omega f,g\right)=(f,g)_{L^2(\Omega)}.$$

Fix a metric measure space $(K,d,\mu)$ equipped with a regular Dirichlet form $(\mathcal{E},\mathcal{F})$. We say the \emph{strong elliptic Harnack inequality} $(\mathrm{EHI})$ holds, if there are constants $\varepsilon\in(0,1)$ and $C\ge 1$ such that for every ball $B=B(x,r)$ and every non-negative harmonic function $u$ on $B$,
$$\mathop{\mathrm{esup}}\limits_{\varepsilon B}u\le C\mathop{\mathrm{einf}}\limits_{\varepsilon B}u.$$
We say the \emph{weak elliptic Harnack inequality} $(\mathrm{wEH})$ holds, if there are constants $0<\delta\le 1$, $0<\varepsilon_1\le\varepsilon_2<1$ and $C\ge 1$ such that for every ball $B=B(x,r)$ and every non-negative harmonic function $u$ on $B$, there exists $\varepsilon_1r\le r'\le\varepsilon_2r$ satisfying
\begin{equation}\tag*{$(\mathrm{wEH}_\delta)$}
	\fint_{B(x,r')}u^\delta\le C\mathop{\mathrm{einf}}\limits_{B(x,r')}u^\delta.
\end{equation}

Let $W:K\times\mathbb{R}_+\to\mathbb{R}_+$ be a scaling function (for short, a \textit{scale}), that is,

\begin{itemize}
	\item[i)] for every $x\in K$, $W(x,\cdot)$ is a homeomorphism on $[0,\infty)$;
	\item[ii)] there exist constants $C_W\ge 1$ and $0<\beta_1\le\beta_2<\infty$ such that 
	$$C_W^{-1}\left(\frac{R}{r}\right)^{\beta_1}\le\frac{W(x,R)}{W(y,r)}\le C_W\left(\frac{R}{r}\right)^{\beta_2}$$
	for all $0<r\le R<\infty$ and all $x,y\in M$ with $d(x,y)\le R$.
\end{itemize}

We say the \emph{exit time estimate} $(\mathrm{E}_W)$ holds, if there are constants $C>0$ and $\varepsilon,\sigma\in(0,1)$ such that for every ball $B=B(x_0,R)$ with $0<R<\mathrm{diam}(K,d)$, for $\mu$-a.e.\ $y\in B$,
\begin{align}
	G_\mu^B1(y)\le CW(B)&\quad\quad\mbox{if}\quad\quad R<\sigma\cdot\mathrm{diam}(K,d);\tag*{$(\mathrm{E}_\le)$}\\
	G_\mu^B1(y)\ge C^{-1}W(B)&\quad\quad\mbox{if}\quad\quad y\in\varepsilon B.\tag*{$(\mathrm{E}_\ge)$}
\end{align}
In particular, for small balls
\begin{equation}\label{E-weak}
	\frac{\mathrm{einf}_{\varepsilon B}G_\mu^B1}{\mathrm{esup}_BG_\mu^B1}\ge C^{-2}.
\end{equation}

\section{Basic geometry and resistance estimates}

In the sequel we follow Kigami's notions given in \cite{ki}. Consider four contracting maps on $\mathbb{R}^2$:
\begin{align*}
	F_0(x,y)=\left(\frac{2}{9}x+\frac{8}{9\sqrt{3}}y,\frac{2}{3}y\right),\quad\quad F_2(x,y)=\left(\frac{1}{3}x-\frac{2}{3},\frac{1}{3}y-\frac{1}{\sqrt{3}}\right),\\
	F_1(x,y)=\left(\frac{2}{9}x-\frac{8}{9\sqrt{3}}y,\frac{2}{3}y\right),\quad\quad F_3(x,y)=\left(\frac{1}{3}x+\frac{2}{3},\frac{1}{3}y-\frac{1}{\sqrt{3}}\right).
\end{align*}
Let $K$ be the invariant set of the iterative function system $\{F_i\}_{i=0}^3$, that is,
$$K=\bigcup_{i=0}^3F_i(K).$$
See Figure \ref{tree_fract}.

\begin{figure}[htbp]
	\centering\includegraphics[width=0.6\textwidth]{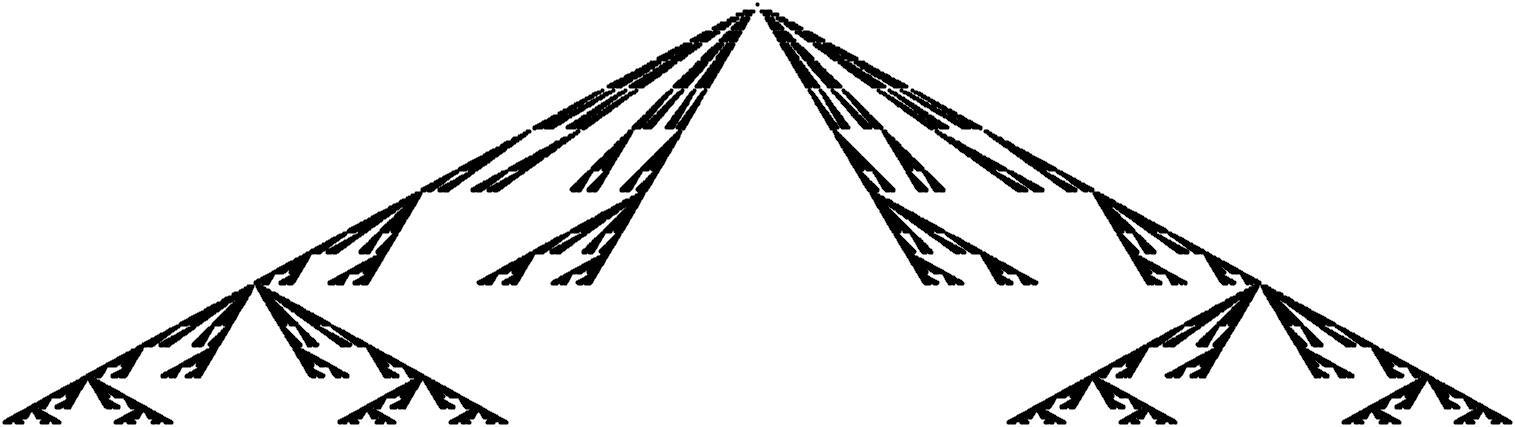}
	\caption{The tree-like fractal $K$.}
	\label{tree_fract}
\end{figure}

For any $\omega=i_1\cdots i_n\in\mathbb{S}_n:=\{0,1,2,3\}^n$, denote $F_\omega=F_{i_1}\circ\cdots\circ F_{i_n}$ and $K_\omega=F_\omega(K)$. Denote $i_1\cdots i_n$ as $i^n$ if every digit $i_j$ is identical to a same $i\in\{0,1,2,3\}$. Set $\mathbb{S}_0=\{\emptyset\},$ $F_\emptyset=\mathrm{id}$ and $K_\emptyset=K$. Let $\mathbb{S}_*=\cup_{i=0}^\infty\mathbb{S}_n$ be the collection of all finite words. Here are some basic geometry facts (where $\omega$, $\tau$ are arbitrary finite words):
\begin{itemize}
	\item[i)] $V_0$, the boundary of $K$, consists of exactly three points
	$$q_1=(0,0),\quad q_2=\left(-\frac{1}{2},-\frac{\sqrt{3}}{2}\right)\quad\mbox{and}\quad q_3=\left(\frac{1}{2},-\frac{\sqrt{3}}{2}\right),$$
	respectively the invariant points of $F_0$ (also $F_1$), $F_2$ and $F_3$.
	\item[ii)] Define inductively $V_{n+1}=\cup_{i=0}^3F_i(V_n)$, then $V_n\subset V_{n+1}$ for all $n\ge 0$, and $V_*:=\cup_{n=0}^\infty V_n$ is dense in $K$.
	\item[iii)] $K$ is simply connected.
	\item[iv)] $K_\omega\subset K_\tau$ if and only if $\tau$ is an ancestor of $\omega$, that is, $\omega=\tau\tau'$ for some $\tau'\in\mathbb{S}_*$.
	\item[v)] say there exist $(x_\omega,y_\omega)\in K_\omega$ and $(x_\tau,y_\tau)\in K_\tau$ such that $x_\omega<x_\tau$, then $K_\omega\cap K_\tau$ is a singleton $q_{\omega,\tau}$ if and only if there exist $\kappa,\omega',\tau'\in\mathbb{S}_*$ such that one of the following three cases holds:
	\begin{itemize}
		\item[a)] $\omega=\kappa2\omega'$, $\tau=\kappa0\tau'$ and $q_{\omega,\tau}=F_\omega(q_1)=F_\tau(q_2)$;
		\item[b)] $\omega=\kappa0\omega'$, $\tau=\kappa1\tau'$ and $q_{\omega,\tau}=F_\omega(q_1)=F_\tau(q_1)$;
		\item[c)] $\omega=\kappa1\omega'$, $\tau=\kappa3\tau'$ and $q_{\omega,\tau}=F_\omega(q_3)=F_\tau(q_1)$.
	\end{itemize}
	\item[vi)] there is no intersection pattern other than described in vi) and v).
	
	\item[vii)] given any $\mathbf{w}:=(w_0,\dotsc,w_3)\in(0,1)^4$ with $\sum_{i=0}^3w_i=1$, there exists a unique self-similar measure $\mu=\mu_\mathbf{w}$ on $K$ such that $\mu(K)=1$, $\mu(V_*)=0$ and
	$$\mu(A)=\sum_{i=0}^3w_i\ \mu\left(F_i^{-1}(A)\right)\quad\mbox{for every Borel set}\quad A\subset K.$$
\end{itemize}

Let $\mathbb{S}_\infty=\{0,1,2,3\}^\mathbb{N}$ be the collection of all infinite words, then there exists a projection map $\pi:\mathbb{S}_\infty\to K$ such that for every $\omega=i_1\cdots i_n\cdots\in\mathbb{S}_\infty$,
$$\{\pi(\omega)\}=\bigcap_{n=1}^\infty K_{i_1\cdots i_n}.$$

It follows from v) that for every point $z\in K$, $\#\ \pi^{-1}(\{z\})=1$ if and only if there exists no $\omega\in\mathbb{S}_*$ such that $z=F_\omega(q_1)$. In this case, let $\omega_z$ be the only element in $\pi^{-1}(\{z\})$. On the contrary, for the case $\#\ \pi^{-1}(\{z\})>1$, define $\omega_z$ in the following way: take a shortest finite word $\omega_*\in\mathbb{S}_*$ such that $z=F_{\omega_*}(q_1)$ (in particular, either $\omega_*=\emptyset$, or the last digit of $\omega_*$ is 2 or 3), and set $\omega_z=\omega_*0^\infty$. In both cases, we call $\omega_z$ the \emph{first address} of $z$.

\subsection{Dirichlet form}

According to v) again, the post-critical set of $K$ is
$$\{0,1\}^\mathbb{N}\cup\{2\}^\mathbb{N}\cup\{3\}^\mathbb{N},$$
which is an infinite set. That is, $K$ is not a p.c.f. fractal. However, Kigami's construction of self-similar Dirichlet forms still works:

\begin{proposition}
	For every $0<s_0<1$ and every self-similar measure $\mu$ on $K$, there exists a unique regular Dirichlet form $(\mathcal{E},\mathcal{F})$ on $L^2(K,\mu)$ such that the following two properties:
	\begin{itemize}
		\item (self-similarity) for every $u\in\mathcal{F}$, with $s_1=s_0$, $s_2=s_3=1-s_0$ we have
		$$\mathcal{E}(u)=\sum\limits_{i=0}^3s_i^{-1}\mathcal{E}(u\circ F_i);$$
		\item $R(q_1,q_2)=R(q_1,q_3)=1$.
	\end{itemize}
\end{proposition}

\begin{proof}
	Take
	$$E_0(u)=(u(q_1)-u(q_2))^2+(u(q_1)-u(q_3))^2\quad\mbox{for all}\quad u:V_0\to\mathbb{R}.$$
	It is easy to verify that, for any $s_i$ as described here, the form
	$$E_1(u)=\sum\limits_{i=0}^3s_i^{-1}E_0(u\circ F_i)\quad\mbox{for all}\quad u:V_1\to\mathbb{R}$$
	satisfies the renormalization equation
	$$E_0(u)=\inf\left\{E_1(v):\ v|_{V_0}=u\right\}.$$
	Set then
	$$\mathcal{E}(u)=\lim\limits_{n\to\infty}E_n(u|_{V_n})\quad\mbox{and}\quad\mathcal{F}=\left\{u\in C(K):\ \mathcal{E}(u)<\infty\right\}.$$
	Same as in \cite[Theorem 3.4.6]{ki} we see $(\mathcal{E},\mathcal{F})$ is a regular Dirichlet form.
	
	We use Lemma \ref{E1} below to show the harmonic extension $u_2$ of $u_2(q_1)=0,u_2(q_2)=1,u_2(q_3)=0$ is the optimal test function for $R(q_2,q_1)$, so that
	$$R(q_2,q_1)=1/\mathcal{E}(u_2)=1/E_0(u_2|_{V_0})=1.$$
	Symmetrically $R(q_3,q_1)=1$.
	
	Finally, uniqueness follows the same idea as in \cite[Theorem 5.1]{L0}.
\end{proof}

From the construction we see $(\mathcal{E},\mathcal{F})$, as a resistance form, is independent of $\mu$. Further, $R$ is a geodesic metric on $K$ generating the topology $\mathcal{T}(K)$, such that $F_i$ is a similitude with ratio $s_i$, and $\mathrm{diam}(K,R)=2$. Additionally, for every two disjoint closed subsets $A,B$ of $K$, there exists a unique function $u_{A,B}$ that is harmonic on $(A\cup B)^c$, $u_{A,B}|_A=1$, $u_{A,B}|_B=0$ and $R(A,B)=1/\mathcal{E}(u_{A,B})$.

On the other hand, as a Dirichlet form, a proper selection of $\mu$ is necessary. We always assume
$$w_0=w_1\quad\mbox{and}\quad w_2=w_3=\frac{1}{2}-w_0$$
so that the Dirichlet space $(K,R,\mu,\mathcal{E},\mathcal{F})$ is symmetric with respect to the reflection over $y$-axis.

\subsection{Three basic harmonic functions}

Given a function $\bar{u}$ on $K\setminus U$, we say $u\in\mathcal{F}$ is an harmonic extension of $\bar{u}$ (into $U$), if $u|_{K\setminus U}=\bar{u}$ and $u$ is harmonic in $U$. By elementary properties of Dirichlet forms, we see such an harmonic extension exists if and only if there exists $u'\in\mathcal{F}$ such that $u'|_{K\setminus U}=\bar{u}$. Further, in this case, the harmonic extension $u$ is unique, and
$$\mathcal{E}(u)=\inf\left\{\mathcal{E}(u'):\ u'\in\mathcal{F},\ u'|_{K\setminus U}=\bar{u}\right\}.$$

For every $x\in K\setminus(\overline{q_1q_2}\cup\overline{q_1q_3})$, clearly there exists a unique point $x_*$ on $\overline{q_1q_2}\cup\overline{q_1q_3}$ such that $x,q_2,q_3$ lie on 3 different connected components of $K\setminus\{x_*\}$. This notions helps us describe an important class of harmonic functions:

\begin{lemma}\label{E1}
	The harmonic extension $u$ of $u(q_i)=a_i$ $(i=1,2,3)$ into $K\setminus V_0$ exists. To be exact, $u(x)=R(q_j,x)a_1+R(q_1,x)a_j$ on $\overline{q_1q_j}$ $(j=2,3)$, and $u(x)=u(x_*)$ for every $x\in K\setminus(\overline{q_1q_2}\cup\overline{q_1q_3})$. Further,
	$$\mathcal{E}(u)=E_0(u|_{V_0})=(a_1-a_2)^2+(a_1-a_3)^2.$$
\end{lemma}

Denote the function $u$ here by $u_{a_2,a_1,a_3}^-$.

\begin{proof}
	By simple calculus.
\end{proof}

Define
\begin{equation}\label{cantor}
	\mathit{C}:=\pi\left(\{2,3\}^\mathbb{N}\right)=\left\{q=(x,y)\in K: y=-\frac{\sqrt{3}}{2}\right\},
\end{equation}
which is clearly a Cantor set. The following harmonic function will be frequently used later:

\begin{lemma}\label{E2}
	The harmonic extension $u$ of $u(q_1)=1$, $u|_\mathit{C}=0$ into $K\setminus(\{q_1\}\cup\mathit{C})$ exists, and $\mathcal{E}(u)=s_0^{-1}+1$.
\end{lemma}

Denote the function $u$ here by $u_\downarrow$.

\begin{proof}
	Assume that $u$ is a desired harmonic extension. To determine $\lambda=u(F_0(q_2))$, we first claim that $u\circ F_2=\lambda u$.
	
	Suppose $u\circ F_2\ne\lambda u$. Observe that
	$$u'=\begin{cases}
		u&\mbox{on}\quad K\setminus K_2\\
		\lambda u\circ F_2^{-1}&\mbox{on}\quad K_2
	\end{cases}$$
	also satisfies the boundary value conditions $u'(q_1)=1$, $u'|_\mathit{C}=0$, and is continuous. Further, by self-similarity of $\mathcal{E}$ there is
	\begin{align*}
		\mathcal{E}(u')=&\ \sum\limits_{i=0,1,3}s_i^{-1}\mathcal{E}(u'\circ F_i)+s_2^{-1}\mathcal{E}(u'\circ F_2)=\sum\limits_{i=0,1,3}s_i^{-1}\mathcal{E}(u\circ F_i)+s_2^{-1}\mathcal{E}(\lambda u)\\
		<&\ \sum\limits_{i=0,1,3}s_i^{-1}\mathcal{E}(u\circ F_i)+s_2^{-1}\mathcal{E}(u\circ F_2)=\mathcal{E}(u),
	\end{align*}
	contradicting the fact that $u$ minimizes energy. Hence $u\circ F_2=\lambda u$.
	
	It follows by symmetry that $u(F_1(q_3))=\lambda$ and $u\circ F_3=\lambda u$. Similarly, $u\circ F_0$ must be of form $u_{\lambda,1,1}^-$, and symmetrically $u\circ F_1$ of form $u_{1,1,\lambda}^-$. As a consequence,
	$$\mathcal{E}(u)=\sum\limits_{i=0}^3s_i^{-1}\mathcal{E}(u\circ F_i)=2s_0^{-1}(1-\lambda)^2+2s_2^{-1}\mathcal{E}(\lambda u)=2s_0^{-1}(1-\lambda)^2+2s_2^{-1}\lambda^2\mathcal{E}(u),$$
	which yields
	\begin{equation}
		\mathcal{E}(u)=\frac{2s_0^{-1}(1-\lambda)^2}{1-2s_2^{-1}\lambda^2}.
	\end{equation}
	The right side attains minimum when $\lambda=\frac{s_2}{2}$, and in this case $\mathcal{E}(u)=\frac{2-s_2}{s_0}=s_0^{-1}+1$. It is easy to check this $u$ is harmonic on $K\setminus(\{q_1\}\cup\mathit{C})$, completing the proof.
\end{proof}

\begin{remark}\label{key}
	The harmonic extension of $u(q_1)=b,u(q_2)=a,u(F_3(q_1))=c$ and $u|_{F_3(\mathit{C})}=0$ is expressed by
	$$u_{a,b,c}^+=\begin{cases}
		u_{s_0a+s_2b,b,b}^-\circ F_0^{-1}&\mbox{on}\quad K_0,\\
		u_{b,b,c}^-\circ F_1^{-1}&\mbox{on}\quad K_1,\\
		u_{a,s_0a+s_2b,s_0a+s_2b}^-\circ F_2^{-1}&\mbox{on}\quad K_2,\\
		cu_\downarrow\circ F_3^{-1}&\mbox{on}\quad K_3.
	\end{cases}$$
	Further, using Lemmas \ref{E1} and \ref{E2}, we obtain
	$$\mathcal{E}\left(u_{a,b,c}^+\right)=(a-b)^2+s_0^{-1}(b-c)^2+s_2^{-1}(s_0^{-1}+1)c^2.$$
\end{remark}

A third important harmonic function is very complicated. As will be used later on, we only consider the case $s_0=\frac{1}{2}$.

\begin{lemma}\label{E3}
	Fix $s_0=\frac{1}{2}$. The harmonic extension $u$ of $u(q_2)=1$, $u(q_1)=0$ and $u|_{F_{0^m23}(\mathit{C})}=0$ for all $m\ge 0$ exists, with $u(F_{0^m2}(q_1))=4^{-1-m}$ and $\mathcal{E}(u)=\frac{3}{2}$.
\end{lemma}

Denote the function $u$ here by $u_\uparrow$.

\begin{proof}
	Assume $u(F_{0^m2}(q_1))=a_m$ and $u(F_{0^m23}(q_1))=a'_m$. Then clearly
	$$u=\begin{cases}
		u_{a_{m-1},a_m,a'_m}^+\circ F_{0^m2}&\mbox{on}\quad K_{0^m2}\quad(m\in\mathbb{N})\\
		0&\mbox{elsewhere}
	\end{cases}$$
	with $a_{-1}=1$. Hence by Remark \ref{key},
	$$\mathcal{E}(u)=\sum\limits_{m=0}^\infty 2^{m+1}\left\{(a_m-a_{m-1})^2+2(a_m-a'_m)^2+6(a'_m)^2\right\}.$$
	Clearly for $a'_m=\frac{1}{4}a_m$ the right side attains minimum, and hence
	$$\mathcal{E}(u)=\sum\limits_{m=0}^\infty 2^{m+1}\left\{(a_m-a_{m-1})^2+\frac{3}{2}a_m^2\right\}.$$
	To minimize the present right side, we require
	$$0=\frac{\partial\mathcal{E}}{\partial a_m}=2^{m+1}\left\{2(a_m-a_{m-1})+3a_m\right\}+2^{m+2}\cdot 2(a_m-a_{m+1}),$$
	that is,
	$$4a_{m+1}-9a_m+2a_{m-1}=0\quad\mbox{for all}\quad m\ge 0.$$
	By elementary calculus, we see
	$$a_m=\alpha\cdot 2^m+\beta\cdot 4^{-m}\quad\mbox{for all}\quad m\ge -1.$$
	Since $a_m\to 0$ as $m\to\infty$ and $a_{-1}=1$, it follows that $a_m=4^{-(m+1)}$, and hence
	\begin{align*}
		\mathcal{E}(u)=&\ \sum\limits_{m=0}^\infty 2^{m+1}\left\{\left(4^{-(m+1)}-4^{-m}\right)^2+\frac{3}{2}\left(4^{-(m+1)}\right)^2\right\}=\frac{3}{2},
	\end{align*}
	proving the desired results.
\end{proof}

\section{Mean exit time: a first attempt}\label{S3}

In the sequel, we fix $s_0=\frac{1}{2}$, and the metric on $K$ is always assigned as the corresponding resistance metric $R(x,y)$. In particular, $\mathcal{E}(u_\downarrow)=3$ and $\mathcal{E}(u_\uparrow)=\frac{3}{2}$.

Let us focus on balls $B_n=B(q_0,2^{-n})$, where $q_0:=F_0(q_2)=F_2(q_1)$. Observe that (see Figure \ref{B-n}) for all $n\ge 1$,
\begin{align*}
	\overline{B_n}=&\ \left(\bigcup\limits_{m=0}^\infty K_{02^{n-1}0^m2}\right)\cup\left(\bigcup\limits_{\omega\in\{0,1\}^{n-1}}K_{2\omega}\right)=:B_n^\uparrow\cup B_n^\downarrow;\\
	\partial B_n=&\ \left(\left\{F_{02^{n-1}}(q_1)\right\}\cup\bigcup\limits_{m=0}^\infty F_{02^{n-1}0^m23}(\mathit{C})\right)\cup\left(\bigcup\limits_{\omega\in\{0,1\}^{n-1}}F_{2\omega}(\mathit{C})\right)=:P_n^\uparrow\cup P_n^\downarrow.
\end{align*}

\begin{figure}[htbp]
	\centering\includegraphics[width=0.2\textwidth]{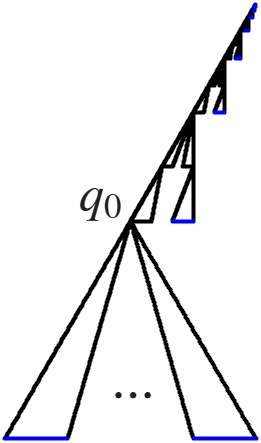}
	\caption{A ball $B_n$, with its boundary in blue}
	\label{B-n}
\end{figure}

We begin by estimating $R(x,B_n^c)$ for every $x\in B_n$ and the corresponding function $\psi_n^x$ such that $\psi_n^x(x)=1$, $\psi_n^x|_{B_n^c}=0$ and $\mathcal{E}(\psi_n^x)=1/R(x,B_n^c)$.

\subsection{Resistance to $B_n^c$}
The first case is quite simple:

\begin{lemma}\label{R0}
	$R(q_0,B_n^c)\asymp 2^{-2n}$ for every $n\ge 1$.
\end{lemma}

\begin{proof}
	Obviously $u=\psi_n^{q_0}$ must be of form
	$$u(x)=\begin{cases}
		u_\downarrow\circ F_{2\omega}^{-1}&\mbox{on}\quad K_{2\omega}\quad\mbox{for each}\quad\omega\in\{0,1\}^{n-1},\\
		u_\uparrow\circ F_{02^{n-1}}^{-1}&\mbox{on}\quad K_{02^{n-1}},\\
		0&\mbox{elsewhere}.
	\end{cases}$$
	Therefore,
	\begin{align*}
		\mathcal{E}(u)=&\ s_{02^{n-1}}^{-1}\mathcal{E}\left(u\circ F_{02^{n-1}}^{-1}\right)+\sum\limits_{\omega\in\{0,1\}^{n-1}}s_{2\omega}^{-1}\mathcal{E}(u\circ F_{2\omega})\\
		=&\ 2^n\mathcal{E}(u_\uparrow)+2^{n-1}\cdot 2^n\mathcal{E}(u_\downarrow)=3\left(2^{n-1}+2^{2n-1}\right),
	\end{align*}
	and the desired result follows directly.
\end{proof}

Now we consider some other typical points in $V_*\cap B_n$: $x_{m,k}=F_{02^{n-1}0^m23^k}(q_1)$ (for all $m,k\in\mathbb{N}$) and $y_k=F_{20^{n-1}2^k}(q_1)$ (for all $k\ge 1$). We will explain later why they are ``typical''.

\begin{lemma}\label{Rmk}
	$R(x_{m,k},B_n^c)\asymp 2^{-(n+m+k)}$ for all $m,k\in\mathbb{N}$, with
	\begin{align*}
		\psi_n^{x_{m,k}}(q_0)\asymp&\ 2^{-(n+m+k)},\\
		\psi_n^{x_{m,k}}(x_{m',0})\asymp&\ 2^{-(m-m'+k)}\quad\mbox{for}\quad 0\le m'\le m,\\
		\psi_n^{x_{m,k}}(x_{m,k'})\asymp&\ 2^{k'-k}\quad\mbox{for}\quad 0\le k'\le k.
	\end{align*}
\end{lemma}

\begin{proof}
	For $u:=\psi_n^{x_{m_0,k_0}}$, assume
	$$u(q_0)=a_{-1,0},\quad u(F_{02^{n-1}0^m23^k}(q_1))=a_{m,k},\quad u(F_{02^{n-1}0^m23^k2}(q_1))=a'_{m,k}.$$
	Then
	$$u=\begin{cases}a_{-1,0}u_\downarrow\circ F_{2\omega}^{-1}&\mbox{on}\quad K_{2\omega}\quad\mbox{for each}\quad\omega\in\{0,1\}^{n-1}\\
		a_{m_0,0}u_\uparrow\circ F_{02^{n-1}0^{m_0+1}}^{-1}&\mbox{on}\quad K_{02^{n-1}0^{m_0+1}}\cap B_n\\
		u_{a_{m-1,0},a_{m,0},a_{m,1}}^+\circ F_{02^{n-1}0^m2}^{-1}&\mbox{on}\quad K_{02^{n-1}0^m2}\quad\mbox{for}\quad m=0,\dotsc,m_0-1\\
		u_{a_{m_0-1,0},a^*_{m_0,0},a^*_{m_0,0}}^-\circ F_{02^{n-1}0^{m_0}22}^{-1}&\mbox{on}\quad K_{02^{n-1}0^{m_0}22}\\
		u_{a^*_{m_0,0},a_{m_0,0},a_{m_0,0}}^-\circ F_{02^{n-1}0^{m_0}20}^{-1}&\mbox{on}\quad K_{02^{n-1}0^{m_0}20}\\
		u_{a_{m_0,0},a_{m_0,0},a_{m_0,1}}^-\circ F_{02^{n-1}0^{m_0}21}^{-1}&\mbox{on}\quad K_{02^{n-1}0^{m_0}21}\\
		u_{a'_{m_0,k},a_{m_0,k},a_{m_0,k}}^-\circ F_{02^{n-1}0^{m_0}23^k0}^{-1}&\mbox{on}\quad K_{02^{n-1}0^{m_0}23^k0}\quad\mbox{for}\quad k=1,\dotsc,k_0-1\\
		u_{a_{m_0,k},a_{m_0,k},a_{m_0,k+1}}^-\circ F_{02^{n-1}0^{m_0}23^k1}^{-1}&\mbox{on}\quad K_{02^{n-1}0^{m_0}23^k1}\quad\mbox{for}\quad k=1,\dotsc,k_0-1\\
		a'_{m_0,k}u_\downarrow\circ F_{02^{n-1}0^{m_0}23^k2}^{-1}&\mbox{on}\quad K_{02^{n-1}0^{m_0}23^k2}\quad\mbox{for}\quad k=1,\dotsc,k_0-1\\
		a''_{m_0,k_0}u_\downarrow\circ F_{02^{n-1}0^{m_0}23^{k_0}}^{-1}&\mbox{on}\quad K_{02^{n-1}0^{m_0}23^{k_0}}\\
		0&\mbox{elsewhere}
	\end{cases}$$
	with $a^*_{m_0,0}:=\frac{1}{2}a_{m_0-1,0}+\frac{1}{2}a_{m_0,0}$ and $a''_{m_0,k_0}=a_{m_0,k_0\vee 1}$. Hence same as in Lemma \ref{R0},
	\begin{align*}
		\mathcal{E}(u)=&\ 2^{n-1}\cdot 2^na_{-1,0}^2\mathcal{E}(u_\downarrow)+2^{n+m_0+1}a_{m,0}^2\mathcal{E}(u_\uparrow)+\sum\limits_{m=0}^{m_0}2^{n+m+1}\mathcal{E}(u_{a_{m-1,0},a_{m,0},a_{m,1}}^+)\\
		&+2^{n+m_0+2}\left\{\mathcal{E}\left(u_{a_{m_0-1,0},a^*_{m_0,0},a^*_{m_0,0}}^-\right)+\mathcal{E}\left(u_{a^*_{m_0,0},a_{m_0,0},a_{m_0,0}}^-\right)+\mathcal{E}\left(u_{a_{m_0,0},a_{m_0,0},a_{m_0,1}}^-\right)\right\}\\
		&+\sum\limits_{k=1}^{k_0-1}2^{n+m_0+k+2}\left\{\mathcal{E}\left(u_{a'_{m_0,k},a_{m_0,k},a_{m_0,k}}^-\right)+\mathcal{E}\left(u_{a_{m_0,k},a_{m_0,k},a_{m_0,k+1}}^-\right)+\left(a'_{m_0,k}\right)^2\mathcal{E}\left(u_\downarrow\right)\right\}\\
		&+2^{n+m_0+k_0+1}\left(a''_{m_0,k_0}\right)^2\mathcal{E}\left(u_\downarrow\right)\\
		=&\ \frac{3}{2}\left(2^{2n}a_{-1,0}^2+2^{n+m_0+1}a_{m_0,0}^2\right)+\sum\limits_{m=0}^{m_0}2^{n+m+1}\left((a_{m-1,0}-a_{m,0})^2+2(a_{m,0}-a_{m,1})^2\right)\\
		&+\sum\limits_{m=0}^{m_0-1}2^{n+m+1}\cdot 6a_{m,1}^2+2^{n+m_0+k_0+1}\cdot 3\left(a''_{m_0,k_0}\right)^2\\
		&+\sum\limits_{k=1}^{k_0-1}2^{n+m_0+k+2}\left\{(a_{m_0,k}-a'_{m_0,k})^2+(a_{m_0,k}-a_{m_0,k+1})^2+3(a'_{m_0,k})^2\right\}.
	\end{align*}
	The way to decide the involved terms $a_{m,k}$ is very close to the proof of Lemma \ref{E3}: by optimizing the right side, we see $a_{m,1}=\frac{1}{4}a_{m,0}$ $(0\le m\le m_0-1)$ and $a'_{m_0,k}=\frac{1}{4}a_{m_0,k}$ $(k=1,\dotsc,(k_0-1)_+)$. Therefore,
	\begin{align*}
		\mathcal{E}(u)=&\ \frac{3}{2}\left(2^{2n}a_{-1,0}^2+2^{n+m_0+1}a_{m_0,0}^2\right)+\sum\limits_{m=0}^{m_0}2^{n+m+1}\left((a_{m-1,0}-a_{m,0})^2+\frac{3}{2}a_{m,0}^2\right)\\
		&+2^{n+m_0+1}\left((a_{m_0-1,0}-a_{m_0,0})^2+2(a_{m_0,0}-a_{m_0,1})^2\right)+2^{n+m_0+k_0+1}\cdot 3\left(a''_{m_0,k_0}\right)^2\\
		&+\sum\limits_{k=1}^{k_0-1}2^{n+m_0+k+2}\left\{(a_{m_0,k}-a_{m_0,k+1})^2+\frac{3}{4}a_{m_0,k}^2\right\}.
	\end{align*}
	By optimizing the right side, we see $a_{m,0}=\alpha\cdot 2^m+\beta\cdot 4^{-m}$ for all $m=-1,\dotsc,m_0$ with proper $\alpha,\beta\in\mathbb{R}$ such that
	$$\frac{\partial\mathcal{E}}{\partial a_{-1,0}}=3\cdot 2^{2n}a_{-1,0}+2\cdot 2^{n+1}(a_{-1,0}-a_{0,0})=0.$$
	
	(1) For $k_0=0$, we further have $a_{m_0,0}=1$ and $a_{m_0,1}=\frac{1}{4}$, so that
	$$a_{m,0}=\frac{24(1+2^{-n})2^m-(3-4\cdot 2^{-n})4^{-m}}{24(1+2^{-n})2^{m_0}-(3-4\cdot 2^{-n})4^{-m_0}}\asymp\begin{cases}
		2^{-m_0-n}&(m=-1)\\
		2^{-m_0+m}&(m=0,\dotsc,m_0)
	\end{cases}$$
	and similarly $a_{m,0}-a_{m-1,0}\asymp 2^{m-m_0}$ for $m=0,\dotsc,m_0$. Thus
	$$R(x_{m,0},B_n^c)^{-1}=\mathcal{E}(u)\asymp\left(2^{2n}\cdot4^{-m_0-n}+2^{n+m_0+1}\right)+\sum\limits_{m=0}^{m_0}2^{n+m+1}\cdot 4^{m-m_0}\asymp 2^{n+m_0}.$$
	
	(2) For $k_0=1$, we further have
	\begin{equation}\label{m+1}
		\frac{\partial\mathcal{E}}{\partial a_{m_0,0}}= 2^{n+m_0}\cdot 2(a_{m_0,0}-a_{m_0-1,0})+2^{n+m_0+1}\cdot 4(a_{m_0,0}-a_{m_0,1})=0,
	\end{equation}
	that is, $a_{m_0,0}=\frac{1}{5}a_{m_0-1,0}+\frac{4}{5}a_{m_0,1}$. Using $a_{m_0,1}=1$, yielded are the results
	$$a_{m,0}=\frac{96(1+2^{-n})2^m-4(3-4\cdot 2^{-n})4^{-m}}{108(1+2^{-n})2^{m_0}-3(3-4\cdot 2^{-n})4^{-m_0}}\asymp\begin{cases}
		2^{-m_0-n}&(m=-1)\\
		2^{-m_0+m}&(m=0,\dotsc,m_0)
	\end{cases}$$
	and $\mathcal{E}(u)\asymp 2^{n+m_0}$.
	
	(3) For $k_0\ge 2$, we further have $a_{m_0,k}=\alpha'\cdot 2^k+\beta'\cdot 4^{-k}$ ($0\le k\le k_0$; in particular, $\alpha'+\beta'=a_{m_0,0}=\alpha\cdot 2^{m_0}+\beta\cdot 4^{-m_0}$) and (\ref{m+1}). Using $a_{m_0,k_0}=1$, yielded are
	\begin{align*}
		a_{m,0}=&\ \frac{24\cdot 2^m-\frac{3-4\cdot 2^{-n}}{1+2^{-n}}\cdot 4^{-m}}{12\cdot 2^{m_0}\left(2^{k_0}+4^{-k_0}\right)-\frac{3-4\cdot 2^{-n}}{1+2^{-n}}\cdot 4^{-m_0-k_0}}\asymp\begin{cases}2^{-n-m_0-k_0}&(m=-1)\\
			2^{m-m_0-k_0}&(m=0,\dotsc,m_0)
		\end{cases}\\
		a_{m_0,k}=&\ \frac{12\cdot 2^{m_0}\left(2^k+4^{-k}\right)-\frac{3-4\cdot 2^{-n}}{1+2^{-n}}\cdot 4^{-m_0-k}}{12\cdot 2^{m_0}\left(2^{k_0}+4^{-k_0}\right)-\frac{3-4\cdot 2^{-n}}{1+2^{-n}}\cdot 4^{-m_0-k_0}}\asymp 2^{k-k_0}
	\end{align*}
	and $\mathcal{E}(u)\asymp 2^{n+m_0+k_0}$.
	
	In summary, the desired results are proved in every case.
\end{proof}

\begin{lemma}\label{Ryk}
	$R(y_k,B_n^c)\asymp 2^{-(n+k)}$ for every $k\in\mathbb{N}$, with
	\begin{align*}
		\psi_n^{y_k}(q_0)\asymp&\ 2^{-(n+k)},\\
		\psi_n^{y_k}(y_{k'})\asymp&\ 2^{k'-k}\quad\mbox{for}\quad 0\le k'\le k.
	\end{align*}
\end{lemma}

\begin{proof}
	Recall that $y_k=F_{20^{n-1}2^k}(q_1)$. For each involved function $u$, let
	$$u(y_k)=b_k,\quad u(F_{20^{n-1}2^k3}(q_1))=b'_k\quad\mbox{for all}\quad k\in\mathbb{N}.$$
	Fix $k_0\ge 1$, then $u=\psi_n^{y_{k_0}}$ is expressed by
	$$u=\begin{cases}b_0u_\uparrow\circ F_{02^{n-1}}^{-1}&\mbox{on}\quad K_{02^{n-1}}\cap B_n\\
		b_0u_\downarrow\circ F_{2\omega}^{-1}&\mbox{on}\quad K_{2\omega}\quad\mbox{for each}\quad\omega\in\{0,1\}^{n-1}\setminus\{0^{n-1}\}\\
		u_{b_{k+1},b_k,b_k}^-\circ F_{20^{n-1}2^k0}^{-1}&\mbox{on}\quad K_{20^{n-1}2^k0}\quad\mbox{for}\quad k=0,\dotsc,k_0-1\\
		u_{b_k,b_k,b'_k}^-\circ F_{20^{n-1}2^k1}^{-1}&\mbox{on}\quad K_{20^{n-1}2^k1}\quad\mbox{for}\quad k=0,\dotsc,k_0-1\\
		b'_ku_\downarrow\circ F_{20^{n-1}2^k3}^{-1}&\mbox{on}\quad K_{20^{n-1}2^k3}\quad\mbox{for}\quad k=0,\dotsc,k_0-1\\
		u_\downarrow\circ F_{20^{n-1}2^{k_0}}^{-1}&\mbox{on}\quad K_{20^{n-1}2^{k_0}}\\
		0&\mbox{elsewhere}
	\end{cases}$$
	with $b_k,b'_k$ ($0\le k\le k_0-1$) optimizing the energy
	-- parallel to Lemma \ref{Rmk}, it follows that $b'_k=\frac{1}{4}b_k$ and
	$$b_k=\frac{3(1+2^{-n})2^k-(3-4\cdot 2^{-n})4^{-k}}{3(1+2^{-n})2^{k_0}-(3-4\cdot 2^{-n})4^{-k_0}}\asymp\begin{cases}2^{-n-k_0}&\mbox{for}\quad k=0;\\
		2^{k-k_0}&\mbox{for}\quad k=1,\dotsc,k_0.
	\end{cases}$$
	The energy estimate follows easily.
\end{proof}

\begin{remark}\label{symm}
	By reflection symmetry, we see the same holds uniformly (identically, to be exact) if $x_{m,k}$ $(m\ge 0,k\ge 1)$ is replaced by $F_{02^{n-1}0^m23\omega}(q_1)$ with arbitrary $\omega\in\{2,3\}^{k-1}$, or $y_k$ ($k\ge 1$) is replaced by $F_{2\omega_1\omega_2}(q_1)$ with arbitrary $\omega_1\in\{0,1\}^{n-1}$ and $\omega_2\in\{2,3\}^k$. These are also ``typical'' points in $B_n$.
\end{remark}

Now we estimate $R(x,B_n^c)$ for non-typical $x$.

\begin{proposition}\label{RPsi}
	For every $x\in B_n$,
	$$R(x,B_n^c)\asymp d_R(x,B_n^c)\wedge\left(R(x,q_0)+2^{-2n}\right).$$
\end{proposition}

\begin{proof}
	Here we split $x\in B_n$ into two parts.
	
	Part 1). those $x\in B_n$ such that there exists $y\in\partial B_n$ such that $x\in\overline{q_0y}$. If $x$ is a typical point, the desired result is already shown; otherwise, let $\omega_x=i_1\cdots i_n\cdots$ be the first address of $x$. Then, up to reflection symmetry, there are two cases:
	
	Case 1. $i_1=0$, so that $\omega_x$ must be of form $02^{n-1}0^m2i_{n+m+2}\cdots$, where $m\ge 0$. If $i_{n+m+1}\in\{0,2\}$, let $z_L=F_{02^{n-1}0^m2}(q_2)$ and $z_R=F_{02^{n-1}0^m2}(q_1)$; if $i_{n+m+1}=1$, let $z_L=F_{02^{n-1}0^m2}(q_1)$ and $z_R=F_{02^{n-1}0^m23}(q_1)$; while if $i_{n+m+1}=3$, let $k$ be the least positive integer such that $i_{n+m+k+1}\in\{0,1\}$, and set $z_L=F_{02^{n-1}0^m23^k}(q_1)$ and $z_R=F_{02^{n-1}0^m23^{k+1}}(q_1)$.
	
	Case 2. $i_1=2$, so that $\omega_x$ must be of form $2\omega_1\omega_2i_{n+m+1}\cdots$, where $\omega_1\in\{0,1\}^{n-1}$, $\omega_2\in\{2,3\}^m$ with some $m\ge 0$, and $i_{n+m+1}\in\{0,1\}$. Let $z_L=F_{2\omega_1\omega_2}(q_1)$ and $z_R=F_{2\omega_1\omega_2i_{n+m+1}}(q_{2+i_{n+m+1}})$.
	
	In both cases $z_L,z_R$ are typical points covered in Remark \ref{symm}, and
	$$R(x,z_L)+R(x,z_R)=R(z_L,z_R).$$
	Now we reduce the resistance network of $\mathcal{E}$ by shorting the whole $B_n^c$ to one point (called $q_2$). This way, we construct a network on $\{x,z_L,z_R,q_2\}$ defined by
	$$E(v)=\sum\limits_{a=L,R}R(x,z_a)^{-1}(v(x)-v(z_a))^2+r_a^{-1}(v(z_a)-v(q_2))^2,$$
	where $r_L,r_R$ are taken so that
	\begin{equation}\label{RLR}
		\frac{1}{R(z_L,B_n^c)}=\frac{1}{r_L}+\frac{1}{R(z_L,z_R)+r_R}\quad\mbox{and}\quad\frac{1}{R(z_R,B_n^c)}=\frac{1}{r_R}+\frac{1}{R(z_L,z_R)+r_L}.
	\end{equation}
	This way,
	\begin{equation}\label{Rx}
		\frac{1}{R(x,B_n^c)}=\frac{1}{R(x,z_L)+r_L}+\frac{1}{R(x,z_R)+r_R},
	\end{equation}
	and $\psi_n^x(z_a)=v(z_a)$ for the equilibrium function $v$ with $v(x)=1$ and $v(q_2)=0$, i.e.
	\begin{equation}\label{Psix}
		v(z_a)=\frac{r_a}{R(x,z_a)+r_a}.
	\end{equation}
	
	From Lemmas \ref{R0}, \ref{Rmk} and \ref{Ryk} we see, whenever $z_L\ne q_0$,
	$$R(z_L,z_R)\asymp R(z_L,B_n^c)\asymp R(z_R,B_n^c)\asymp d_R(x,\partial B_n).$$
	Therefore, by (\ref{RLR}), $r_L,r_R\gtrsim R(z_R,B_n^c)$ and at least one of them takes ``$\asymp$''. Hence $R(x,B_n^c)\asymp R(z_R,B_n^c)\asymp 2^{-n}-R(x,q_0)$ by (\ref{Rx}).
	
	On the other hand, for $z_L=q_0$, then $z_R$ is either $x_{0,0}$ or $y_1$ (up to reflection symmetry). This way, $R(z_L,B_n^c)\asymp 2^{-2n}$, $R(z_R,B_n^c)\asymp 2^{-n-1}$ and $R(z_L,z_R)=2^{-n-1}$. Thus by (\ref{RLR}), $r_L\asymp 2^{-2n}$ and $r_R\gtrsim 2^{-n}$, yielding by (\ref{Rx}) that
	$$R(x,B_n^c)\asymp \left((R(x,z_L)+r_L)^{-1}+O(2^{-n})\right)^{-1}\asymp R(x,z_L)+r_L\asymp R(x,q_0)+2^{-2n}.$$
	The desired result follows directly.
	
	Part 2). the remaining $x$. Take $\bar{x}$ as the point of Part 1) that is closest to $x$, and $\psi_{x,\bar{x}}$ as the optimal function for the effective resistance $R(x,\bar{x})$. By geometry we see
	\begin{equation}\label{paste}
		R(x,q_0)=R(x,\bar{x})+R(\bar{x},q_0)\quad\mbox{and}\quad d_R(x,B_n^c)=R(x,\bar{x})+R(\bar{x},B_n^c).
	\end{equation}
	Further, let $V_x$ be the component of $B_n\setminus\{\bar{x}\}$ containing $x$. We can easily show that $\psi_n^{\bar{x}}=1$ on $V_x$, $\psi_{x,\bar{x}}=0$ on $V_x^c$, while
	\begin{equation}\label{psiIn}
		\psi_n^x=\begin{cases}
			\psi_n^x(\bar{x})\psi_n^{\bar{x}}&\mbox{on}\quad V_x^c;\\
			\psi_n^x(\bar{x})+\left(1-\psi_n^x(\bar{x})\right)\psi_{x,\bar{x}}&\mbox{on}\quad V_x.
		\end{cases}
	\end{equation}
	Therefore, by assuming $\psi_n^x(\bar{x})=:a$, we have
	\begin{align*}
		R\left(x,B_n^c\right)^{-1}=&\ \mathcal{E}\left(\psi_n^x\right)=\mathcal{E}\left(a\psi_n^{\bar{x}}\right)+\mathcal{E}\left(a+(1-a)\psi_{x,\bar{x}}\right)\\
		=&\ a^2\mathcal{E}\left(\psi_n^{\bar{x}}\right)+(1-a)^2\mathcal{E}\left(\psi_{x,\bar{x}}\right)=a^2R\left(\bar{x},B_n^c\right)^{-1}+(1-a)^2R(x,\bar{x})^{-1}.
	\end{align*}
	Since $R(x,B_n^c)$ can only be attained for the optimal value $a$, it follows that
	\begin{equation}\label{cut}
		a=\psi_n^x(\bar{x})=\frac{R\left(\bar{x},B_n^c\right)}{R(x,\bar{x})+R\left(\bar{x},B_n^c\right)}\quad\mbox{and}\quad R\left(x,B_n^c\right)=R(x,\bar{x})+R\left(\bar{x},B_n^c\right).
	\end{equation}
	Combining (\ref{paste}) and the estimate for $\bar{x}$ given in Part 1), it follows that
	\begin{align*}
		R(x,B_n^c)=&\ R(x,\bar{x})+R(\bar{x},B_n^c)\asymp R(x,\bar{x})+d_R(\bar{x},B_n^c)\wedge\left(R(\bar{x},q_0)+2^{-2n}\right)\\
		=&\ \left(R(x,\bar{x})+d_R(\bar{x},B_n^c)\right)\wedge\left(R(x,\bar{x})+R(\bar{x},q_0)+2^{-2n}\right)\\
		=&\ d_R(x,B_n^c)\wedge\left(R(x,q_0)+2^{-2n}\right),
	\end{align*}
	which is exactly what we desire.
\end{proof}

Note in this proof that, if $\omega_x=i_1i_2\cdots$, then $\bar{x}=F_{i_1\cdots i_{m_x}}(q_1)$, where
\begin{itemize}
	\item $m_x=m$, if $i_1\cdots i_m=02^{n-1}0^{m-n-1}2$ and $i_m\in\{0,1\}$;
	\item $m_x=m+m'$, if $i_1\cdots i_{m+1}=02^{n-1}0^{m-n-1}2^2$, $i_{m+2},\cdots,i_{m+m'}\in\{0,2\}$ and $i_{m+m'+1}\in\{1,3\}$;
	\item $m_x=m+m'+m''$, if the following are true:
	\begin{itemize}
		\item[a)] either $i_1\cdots i_{m+1}=02^{n-1}0^{m-n-1}23$, or $i_1= 2,i_2\cdots i_{m+1}=\{0,1\}^{n-1}$,
		\item[b)] $i_{m+2},\cdots,i_{m+m'}\in\{2,3\}$,
		\item[c)] either $i_{m+m'+1}=0,i_{m+m'+2},\cdots i_{m+m'+m''}\in\{0,2\}^{m''-1}$, $i_{m+m'+m''+1}\in\{1,3\}$ or $i_{m+m'+1}=1,i_{m+m'+2},\cdots i_{m+m'+m''}\in\{1,3\}^{m''-1}$, $i_{m+m'+m''+1}\in\{0,2\}$.
	\end{itemize}
\end{itemize}
From this determination of $\bar{x}$, we can easily show that
\begin{equation}\label{InIn}
	R(x,\bar{x})\le 2^{-n-1}\le d_R(\bar{x},B_n^c)\le 2^{-n}\quad\mbox{and}\quad R(\bar{x},q_0)\le 2^{-n-1}
\end{equation}
if $x\in\frac{1}{2}B_n$, while
\begin{equation}\label{InOut}
	R(x,\bar{x})\le 2^{-n-k}\asymp d_R(\bar{x},B_n^c)\quad\mbox{and}\quad R(\bar{x},q_0)\asymp 2^{-n}
\end{equation}
if $x\in(1-2^{-k-1})B_n\setminus(1-2^{-k})B_n$ (with any $k\ge 1$). It follows in either case that
$$d_R(x,B_n^c)\asymp 2^{-n}-R(x,q_0),$$
which improves the resistance estimate to some extent.

Further, this observation also helps us estimate $\psi_n^x$. To be exact, for $x$ in Part 1), if $z_L\ne q_0$, then $\psi_n^x(z_a)\asymp 1$. Thus by comparison technique on $B_n\setminus\overline{z_Lz_R}$ we see
\begin{equation}\label{psiL}
	\psi_n^x\asymp\psi_n^{z_R};
\end{equation}
while if $z_L=q_0$, then
\begin{equation}\label{psiR}
	\psi_n^x(q_0)\asymp (2^{-2n}R(x,q_0)^{-1})\wedge 1\quad\mbox{and}\quad\psi_n^x(z_R)\asymp 1.
\end{equation}
The value of $\psi_n^x$ elsewhere can then be estimated from its piecewise harmonic structure.

Meanwhile, for $x$ in Part 2), combining (\ref{psiIn}), (\ref{cut}) with (\ref{InIn}) or (\ref{InOut}), we see
\begin{equation}\label{psiOut}
	\psi_n^x\asymp\psi_n^{\bar{x}}.
\end{equation}

\subsection{Measure of balls and exit time estimates}

Fix an arbitrary self-similar measure $\mu$ with $w_0=w_1$, $w_2=w_3$. Recall that $\mu$ is said to be doubling, if there exists a constant $C\ge 1$ such that for every $x\in K$ and $r>0$,
\begin{equation}\label{c-doub}
	\mu(B(x,2r))\le C\mu(B(x,r)).
\end{equation}

\begin{lemma}\label{v-doub}
	$\mu$ is never doubling, but (\ref{c-doub}) always holds for $x\in V_n$ and $r\le 2^{-n-1}$.
\end{lemma}

\begin{proof}
	Let $y_n=F_{20^{n-1}}(q_2)$ and $r_n=2^{-n}$. Then for sufficiently large $n$,
	\begin{align*}
		\frac{\mu(B(y_n,2r_n))}{\mu(B(y_n,r_n))}=&\ \frac{\mu\left(\bigcup\limits_{\omega\in\{0,1\}^{n-1}\setminus\{0^{n-1}\}}K_{2\omega}\cup K_{20^{n-2}2}\cup K_{20^{n-3}20}\cup K_{20^{n-3}21}\cup\bigcup\limits_{m=0}^\infty K_{02^{n-1}0^m2}\right)}{\mu\left(K_{20^{n-2}2}\cup\bigcup\limits_{m=0}^\infty K_{20^{n+m-1}2}\right)}\\
		=&\ \frac{1-w_0}{2}w_0w_2^{-1}\left\{2^n+w_0^{-2}(4w_2-w_0)\right\}+w_0^3w_2^{-2}\left(\frac{w_2}{w_0}\right)^n\ge 2^n\frac{1-w_0}{4w_2w_0^{-1}},
	\end{align*}
	contradicting (\ref{c-doub}) directly. Hence $\mu$ is never doubling.
	
	Now we focus on the case $x\in V_n$. Fix $k\ge 1$ such that $2^{-n-k-1}<r\le 2^{-n-k}$.
	
	If $x=F_{i_1\cdots i_n}(q_1)$, with $i_{n-m}\in\{2,3\}$ and $i_{n-m+1},\dotsc,i_n\in\{0,1\}$ with some $0\le m<n$, then $B(x,2r)$ is contained in exactly $1+2^m$ cells of order $n+k-1$: that is, $K_{i_1\cdots (i_{n-m}-2)i_{n,m}^{m+k-1}}$ and $K_{i_1\cdots i_{n-m}\omega}$ (where $\omega\in\{0,1\}^{m+k-1}$), and $B(x,r)$ contains parallel cells of order $n+k+2$. Consequently,
	$$\frac{\mu(B(x,2r))}{\mu(B(x,r))}\le\frac{\mu\left(K_{i_1\cdots (i_{n-m}-2)i_{n,m}^{m+k-1}}\cup\bigcup\limits_{\omega\in\{0,1\}^{m+k-1}}K_{i_1\cdots i_{n-m}\omega}\right)}{\mu\left(K_{i_1\cdots (i_{n-m}-2)i_{n,m}^{m+k+2}}\cup\bigcup\limits_{\omega\in\{0,1\}^{m+k+1}}K_{i_1\cdots i_{n-m}\omega}\right)}\le(w_0\wedge w_2)^{-3}.$$
	If the expression of $x$ above is impossible, then either $x=q_1$ so that the first term disappears, or $x$ is on a bottom so that the second group of terms disappear. Thus the same inequality holds more easily, implying (\ref{c-doub}) in every case.
\end{proof}

\begin{lemma}
	$\varepsilon_0\le\int_Ku_\downarrow d\mu\le4\varepsilon_0$ and $\varepsilon_1\le\int_Ku_\uparrow d\mu\le 4\varepsilon_1$ with constants $\varepsilon_0,\varepsilon_1>0$ depending only on $w_0$.
\end{lemma}

\begin{proof}
	Recall by Lemma \ref{E2} that $u_\downarrow(F_\omega(q_1))=4^{-m}$ for every $\omega\in\{2,3\}^m$ (where $m\in\mathbb{N}$). Since $u_\downarrow$ is harmonic on every piece $K_\omega$, then by comparison technique
	$$4^{-m-1}\le u_\downarrow\le 4^{-m}\quad\mbox{on}\quad K_{\omega0}\cup K_{\omega1}.$$
	Hence it is easy to show $\varepsilon_0\le\int_Ku_\downarrow d\mu\le4\varepsilon_0$ with $\varepsilon_0=\frac{w_0}{2-w_2}$.
	
	Similarly, using Lemma \ref{E3} and comparison technique, we see $4^{-m}\le u_\uparrow\le 4^{1-m}$ on $K_{0^m20}\cup K_{0^m22}$, $4^{-m-1}\le u_\uparrow\le 4^{-m}$ on $K_{0^m21}$, and $u_\uparrow=4^{-m-1}u_\downarrow\circ F_{0^m23}^{-1}$ on $K_{0^m23}$; further, $u\equiv 0$ elsewhere. Thus similar to the case $u_\downarrow$ we have
	\begin{align*}
		\int_Ku_\uparrow d\mu=&\ \sum\limits_{m=0}^\infty\int_{K_{0^m2}}u_\uparrow d\mu=\sum\limits_{m=0}^\infty\left(\sum\limits_{i=0}^2\int_{K_{0^m2i}}u_\uparrow d\mu+4^{-m-1}w_{0^m23}\int_{K_{0^m23}}u_\downarrow d\mu\right)\\
		\in&\ \left(\sum\limits_{m=0}^\infty4^{-m}(w_{0^m20}+w_{0^m22})+4^{-m-1}w_{0^m21}+4^{-m-1}w_{0^m23}\varepsilon_0,\right.\\
		&\quad\quad\left.\sum\limits_{m=0}^\infty4^{1-m}(w_{0^m20}+w_{0^m22})+4^{-m}w_{0^m21}+4^{-m-1}w_{0^m23}4\varepsilon_0\right)\\
		=&\ \left(\frac{w_2}{4-w_0}(2+w_0+w_2\varepsilon_0),\ \frac{4w_2}{4-w_0}(2+w_0+w_2\varepsilon_0)\right).
	\end{align*}
	The proof is then completed by letting $\varepsilon_1=\frac{w_2}{4-w_0}(2+w_0+w_2\varepsilon_0)$.
\end{proof}

Recall by \cite[Subsection 6.5]{ghl14} that
\begin{equation}\label{GRPs}
	G_\mu^{B_n}1(x)=R(x,B_n^c)\int_{B_n}\psi_n^xd\mu\quad\mbox{for all}\quad x\in B_n.
\end{equation}
Combining the lemma above, we can look for exit time estimates. For the first step, we give the general two-sided estimate on $G_\mu^{B_n}1$:

\begin{theorem}\label{main-G}
	For all $x\in B_n$,
	\begin{equation}\label{exit}
		G_\mu^{B_n}1(x)\asymp\begin{cases}
			\left(2^nd_R(x,B_n^c)\right)^2\left(2^{-n}w_0^n+w_2^n\left(R(x,q_0)+2^{-2n}\right)\right),&\mbox{for}\quad x\in B_n^\uparrow;\\
			\left(2^nd_R(x,B_n^c)\right)^2\left(2^{-n}w_0^n+2^{-2n}w_2^n\right),&\mbox{for}\quad x\in B_n^\downarrow.
		\end{cases}
	\end{equation}
\end{theorem}

\begin{proof}
	Recalling the construction of $\psi_n^{q_0}$ (see Lemma \ref{R0}), we see
	$$(\varepsilon_0\wedge\varepsilon_1)\mu(B_n)\le\int_{B_n}\psi_n^{q_0}d\mu\le 4(\varepsilon_0\vee\varepsilon_1)\mu(B_n),$$
	which yields the desired estimate for $x=q_0$ with the help of Lemma \ref{R0}.
	
	Repeating the same idea, we obtain (using Lemma \ref{Ryk})
	\begin{align}
		\int_{B_n}\psi_n^{y_k}d\mu\asymp&\ w_{20^{n-1}2^{k+1}}\varepsilon_0+\sum\limits_{k'=0}^kb_{k'}w_{20^{n-1}2^{k'}}+(2^n-1)b_0w_{20^{n-1}}\varepsilon_0+b_0w_{02^{n-1}}\varepsilon_1\notag\\
		\asymp&\ w_0^n2^{-k}\sum\limits_{k'=0}^k(2w_2)^{k'}+2^{-n-k}\mu(B_n)\asymp 2^{-k}\left(2^{-n}w_2^n+w_0^n\right);\label{E-down}
	\end{align}
	similarly, (using Lemma \ref{Rmk} and the fact $w_0<\frac{1}{2}$)
	\begin{align}
		\int_{B_n}\psi_n^{x_{m,k}}d\mu\asymp&\ 2^na_{-1,0}w_{20^{n-1}}+\sum\limits_{m'=0}^ma_{m',0}w_{02^{n-1}0^{m'}}+a_{m,0}w_{02^{n-1}0^m}\varepsilon_1\notag\\
		&\ +\sum\limits_{k'=0}^ka_{m,k'}w_{02^{n-1}0^m23^{k'}}+w_{02^{n-1}0^m23^k}\varepsilon_0\notag\\
		\asymp&\ 2^{-(n+m+k)}\mu(B_n^\downarrow)+\sum\limits_{m'=0}^m2^{m'-m-k}w_2^nw_0^{m'}+\sum\limits_{k'=0}^k2^{k'-k}w_2^{n+k'}w_0^m\notag\\
		\asymp&\ 2^{-(n+m+k)}\cdot 2^nw_0^n+2^{-k}w_2^n\left(2^{-m}+w_0^m\right)\asymp 2^{-(m+k)}\left(w_0^n+w_2^n\right).\label{E-up}
	\end{align}
	These two estimates passage to every $x\in B_n\setminus\frac{1}{2}B_n$ according to (\ref{psiL}) and (\ref{psiOut}). Thus, the desired estimate on $G^{B_n}_\mu1(x)$ follows by (\ref{GRPs}) and Proposition \ref{RPsi} in this case.
	
	Now we focus on $x\in\frac{1}{2}B_n\setminus\{q_0\}$. Since $z_L=q_0$ (for $\bar{x}$), then by (\ref{psiR}), (\ref{psiOut}) and piecewise harmonicity, we see for $x\in K_{02^{n-1}}$ that
	\begin{align*}
		\int_{B_n}\psi_n^xd\mu=&\ \left(\int_{K_{02^n}}+\int_{K_{02^{n-1}0}\cap B_n}+\int_{B_n^\downarrow}\right)\psi_n^xd\mu\\
		\asymp&\ (w_2+2w_0+\varepsilon_0w_2)w_{02^n}+\varepsilon_1w_{02^{n-1}0}+\mu(B_n^\downarrow)\varepsilon_0\psi_n^x(q_0)\\
		\asymp&\ w_2^n+2^nw_0^n\left\{\left(2^{-2n}R(x,q_0)^{-1}\right)\wedge 1\right\};
	\end{align*}
	while for $x\in K_{2\omega}$ (say $\omega=0^{n-1}$),
	\begin{align*}
		\int_{B_n}\psi_n^xd\mu=&\ \left(\int_{K_{20^{n-1}}}+\sum\limits_{\omega'\in\{0,1\}^{n-1}\setminus\{0^{n-1}\}}\int_{K_{2\omega'}}+\int_{B_n^\uparrow}\right)\psi_n^xd\mu\\
		\asymp&\ w_{20^n}+\left\{\left(2^{n-1}-1\right)w_{20^{n-1}0}+\mu(B_n^\uparrow)\right\}\varepsilon_0\psi_n^x(q_0)\\
		\asymp&\ w_0^n+(2^nw_0^n+w_2^n)\left\{\left(2^{-2n}R(x,q_0)^{-1}\right)\wedge 1\right\}.
	\end{align*}
	The desired estimate follows by combining Proposition \ref{RPsi}.
\end{proof}

Note in particular that, since $w_i\le\frac{1}{2}$, the following upper estimate follows:
\begin{equation}\label{E-out}
	G_\mu^{B_n}1(x)\lesssim d_R(x,B_n^c)\mu(B_n),
\end{equation}
which is sharp if and only if $w_0=2w_2$.

By the way, we can introduce $A_n^x,D_n^x\subset B_n$ as follows:
\begin{itemize}
	\item let $A_n^{q_0}=D_n^{q_0}=B_n$;
	\item for $x\in B_n^\uparrow\setminus\{q_0\}$, let $A_n^x=B_n^\uparrow$ and $D_n^x=B_n^\downarrow$;
	\item for $x\in B_n^\downarrow\setminus\{q_0\}$, let $A_n^x$ be the connected component of $B_n\setminus\{q_0\}$ containing $x$ and $D_n^x=B_n$.
\end{itemize}

Based on these notions, this estimate can also be expressed as
$$G_\mu^{B_n}1(x)\asymp\left(R(x,q_0)+2^{-2n}\right)\left(\mu(A_n^x)+\mu(D_n^x)\left\{\left(2^{-2n}R(x,q_0)^{-1}\right)\wedge 1\right\}\right).$$

\begin{theorem}\label{E-no}
	For $w_2>w_0$, condition $(\mathrm{E}_W)$ is impossible under any scale $W$.
\end{theorem}

\begin{proof}
	By (\ref{exit}) we clearly see for any $0<\varepsilon<1$ that
	$$\inf_{\varepsilon B_n}G_\mu^{B_n}1\le\inf_{(\varepsilon\wedge 4^{-n})B_n}G_\mu^{B_n}1\asymp 4^{-n}\mu(B_n)=4^{-n}\left(w_0w_2^{n-1}+2^nw_2w_0^{n-1}\right)\asymp 4^{-n}w_2^n+2^{-n}w_0^n.$$
	Meanwhile, using (\ref{E-up}) we see
	$$\sup_{B_n}G_\mu^{B_n}1\ge G_\mu^{B_n}1(x_{1,0})=R(x_{1,0},B_n^c)\int_{B_n}\psi_n^{x_{1,0}}d\mu\asymp 2^{-n}\left(w_0^n+w_2^n\right).$$
	Therefore,
	$$\frac{\inf_{\varepsilon B_n}G_\mu^{B_n}1}{\sup_{B_n}G_\mu^{B_n}1}\lesssim\frac{4^{-n}w_2^n+2^{-n}w_0^n}{2^{-n}w_0^n+2^{-n}w_2^n}=\frac{1+2^{-n}\left(\frac{w_2}{w_0}\right)^n}{1+\left(\frac{w_2}{w_0}\right)^n}\to\begin{cases}0,&\mbox{if}\quad w_2>w_0;\\
		1,&\mbox{if}\quad w_2\le w_0.\end{cases}$$
	Hence (\ref{E-weak}) fails in the case $w_2>w_0$. In particular, $(\mathrm{E}_W)$ is impossible with any scale $W$.
\end{proof}

Here we can not exclude exit time estimates in the case $w_2\le w_0$. This will be discussed in a subsequently paper.

\section{Strong and weak Harnack inequalities}\label{S4}

For any $m\ge 0$, $n,k\ge 1$ and $\omega\in\{0,1\}^{n-1}$, let $\phi_{n,m,k}$ be the harmonic function $\phi$ on $B_n$ with boundary condition $\phi|_{\partial B_n}=1_{F_{02^{n-1}0^m23^k}(\mathit{C})}$, and $\phi'_{n,\omega,k}$ be the harmonic function $\phi$ on $B_n$ with boundary condition $\phi|_{\partial B_n}=1_{F_{2\omega2^k}(\mathit{C})}.$ The estimates below follow the same calculus as in Lemmas \ref{Rmk} and \ref{Ryk}, for which reason we omit the details:
\begin{equation}\label{phi}\begin{split}
		\phi_{n,m,k}\asymp&\ 2^{-n-m-k}u_\downarrow\circ F_{2\omega}^{-1}\quad\mbox{on}\quad K_{2\omega},\quad\phi_{n,m,k}(x_{0,0})\asymp 2^{1-m-k};\\
		\phi'_{n,\omega,k}\asymp&\ 2^{-n-k}u_\uparrow\circ F_{02^{n-1}}^{-1}\quad\mbox{on}\quad B_n^\uparrow,\quad\ \ \phi'_{n,\omega,k}(y_1)\asymp 2^{1-k};\\
		\phi'_{n,\omega,k}\asymp&\ 2^{-n-k}u_\downarrow\circ F_{2\omega'}^{-1}\quad\ \ \mbox{on}\quad K_{2\omega'}\quad\mbox{with}\quad \omega'\in\{0,1\}^{n-1}\setminus\{\omega\}.
	\end{split}
\end{equation}

\begin{theorem}
	Condition $(\mathrm{EHI})$ is impossible for $(\mathcal{E},\mathcal{F})$ on $K$ (regardless of $\mu$).
\end{theorem}

\begin{proof}
	With the help of Lemma \ref{E1}, it follows from (\ref{phi}) that for any $0<\varepsilon\le\frac{1}{2}$,
	$$\inf\limits_{\varepsilon B_n}\phi'_{n,\omega,k}\asymp 2^{-n-k}\quad\mbox{and}\quad\sup\limits_{\varepsilon B_n}\phi'_{n,\omega,k}\asymp\left(2^{1-k}\cdot 2\varepsilon\right)\vee2^{-n-k}$$
	for sufficiently large $n$. That is,
	$$\frac{\inf_{\varepsilon B_n}\phi'_{n,\omega,k}}{\sup_{\varepsilon B_n}\phi'_{n,\omega,k}}\asymp\left(2^n\varepsilon+1\right)^{-1}\to 0\quad\mbox{as}\quad n\to\infty,$$
	showing that $(\mathrm{EHI})$ always fails.
\end{proof}

\begin{theorem}
	For any $0<\delta\le 1$, $(\mathrm{wEH}_\delta)$ holds if and only if $w_2\le 2^{1-\delta}w_0$.
\end{theorem}

\begin{proof}
	We consider first the balls $B_n$. Note that for every non-negative harmonic function $u$ on $B_n$, there exist non-negative continuous functions $v',v''$ on $\partial B_n$, satisfying $v'=0$ on $P_n^\downarrow$ and $v''=0$ on $P_n^\uparrow$, such that $u=u'+u''$ on $B_n$, where $u',u''$ are respectively harmonic extensions of $v',v''$: actually, $v'=u|_{P_n^\uparrow},v''=u|_{P_n^\downarrow}$.
	
	Further, let $\nu_m$ and $\nu_\omega$ be the $\log_32$-dimensional Hausdorff measures respectively normalized on $P_{n,m}:=F_{02^{n-1}0^m23}(\mathit{C})$ and $P_{n,\omega}:=F_{2\omega}(\mathit{C})$ (where $m\in\mathbb{N}$ and $\omega\in\{0,1\}^{n-1}$), then using monotone convergence with step functions together with reflection symmetry, we see from (\ref{phi}) that if
	$$\int_{P_{n,m}}v'd\nu_m=u_m\quad\mbox{for each}\quad m\ge 0\quad\mbox{and}\quad\int_{P_{n,\omega}}v''d\nu_\omega=u_\omega,$$
	then
	$$\inf_{\frac{1}{2}B_n}u'\asymp u'(q_0)\asymp\sum\limits_{m=0}^\infty 2^{-n-m}u_m,\quad\inf_{\frac{1}{2}B_n}u''\asymp u''(q_0)\asymp\sum\limits_{\omega\in\{0,1\}^{n-1}}2^{-n}u_\omega;$$
	while
	\begin{align*}
		\fint_{\frac{1}{2}B_n}(u')^\delta d\mu\asymp&\ \frac{1}{w_2^n+(2w_0)^n}\left\{w_2^n\left(\sum\limits_{m=0}^\infty2^{-m}u_m\right)^\delta+\sum\limits_{\omega\in\{0,1\}^{n-1}}w_0^n\left(\sum\limits_{m=0}^\infty2^{-2-n-m}u_m\right)^\delta\right\}\\
		\asymp&\ \frac{2^{\delta n}w_2^n+(2w_0)^n}{w_2^n+(2w_0)^n}\left(\sum\limits_{m=0}^\infty2^{-n-m}u_m\right)^\delta\asymp\left\{\left(2^{\delta n}\wedge\left(\frac{w_2}{2^{1-\delta}w_0}\right)^n\right)\vee 1\right\}\left(\inf_{\frac{1}{2}B_n}u'\right)^\delta
	\end{align*}
	and
	\begin{align*}
		\fint_{\frac{1}{2}B_n}u''d\mu\asymp&\ \frac{w_2^n\sum\limits_{\omega\in\{0,1\}^{n-1}}2^{-n}u_\omega+\sum\limits_{\omega\in\{0,1\}^{n-1}}w_0^n\left(u_\omega+\sum\limits_{\omega'\in\{0,1\}^{n-1}\setminus\{\omega\}}2^{-n}u_{\omega'}\right)}{w_2^n+(2w_0)^n}\\
		=&\ \frac{w_2^n+\left(\frac{3}{2}\cdot 2^n-1\right)w_0^n}{w_2^n+(2w_0)^n}\sum\limits_{\omega\in\{0,1\}^{n-1}}2^{-n}u_\omega\asymp\inf_{\frac{1}{2}B_n}u''.
	\end{align*}
	It follows that $u''$ always satisfies $(\mathrm{wEH}_1)$ (and $(\mathrm{wEH}_\delta)$ with any $0<\delta\le 1$ follows trivially); while $u'$ satisfies $(\mathrm{wEH}_\delta)$ if and only if $w_2\le2^{1-\delta}w_0$. Consequently, $(\mathrm{wEH}_\delta)$ fails for $w_2<2^{1-\delta}w_0$, while it is proved on $B_n$ for $w_2\le2^{1-\delta}w_0$ since
	$$\fint_{\frac{1}{2}B_n}u^\delta d\mu\le C_\delta\left\{\fint_{\frac{1}{2}B_n}(u')^\delta d\mu+\fint_{\frac{1}{2}B_n}(u'')^\delta d\mu\right\}\lesssim\left(\inf_{\frac{1}{2}B_n}u'\right)^\delta+\left(\inf_{\frac{1}{2}B_n}u''\right)^\delta\le 2\left(\inf_{\frac{1}{2}B_n}u\right)^\delta.$$
	
	For any other point $q=F_{i_1\cdots i_n}(q_j)\in V_n$ (with $j=1$ whenever possible), set $n_q=\#\{k:\ i_k=0\ \mbox{or}\ 1\}$ if $j=1$, and $n_q=-\infty$ otherwise. Then, the corresponding ball $B(q,2^{-n})$ is covered by $1_{n_q<n}$ upper and $2^{-n_q}$ lower cells of $n$-th order. Thus the above is adapted to $B(q,2^{-n})$ by replacing $2^n$ everywhere\footnote{Of course this refers to the terms coming after ``$\asymp$''; on the contrary, the term $3\cdot 2^n-2$, etc., counting the number of ``other subcells below $q_0$'', is replaced by 0 as an entity.} by $2^{n_q}$. Since $n_q\le n$, then $(\mathrm{wEH}_\delta)$ holds on every $B(q,2^{-n})$ with $q\in V_n$ whenever $w_2\le2^{1-\delta}w_0$.
	
	Now we consider a general ball $B=B(x_0,r_0)$ in $K$. Take the minimal $n_0\in\mathbb{N}$ such that $3\cdot 2^{-n_0}\le r_0$. In particular,
	\begin{equation}\label{n-0}
		6^{-1}r_0<2^{-n_0}\le 3^{-1}r_0.
	\end{equation}
	By the geometry of $K$, there are two possibilities (see Figure \ref{B-n-q}):
	
	\begin{figure}[htbp]
		\centering\includegraphics[width=0.3\textwidth]{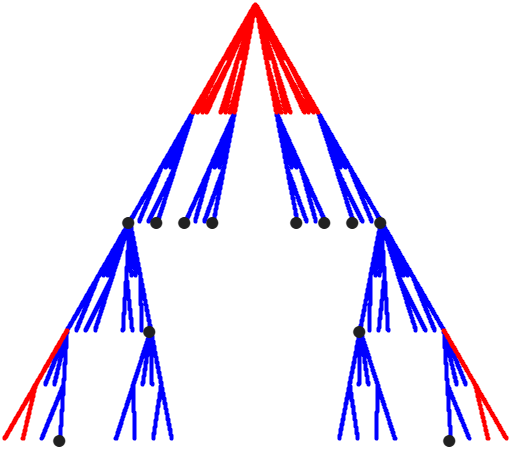}
		\caption{A cell $K_\omega$ splits to possibilities 1 (in red) and 2 (in blue), with lattice points $q'$ in bald}
		\label{B-n-q}
	\end{figure}
	
	\begin{itemize}
		\item the point $q'$ in $V_{n_0}$ closest to $x_0$ satisfies $R(q',x_0)\le 2^{-2-n_0}$. This case, using the fact that $u$ is harmonic on
		$$B\supset B(q',r_0-R(q',x_0))\supset B(q',2(r'+R(q',x_0))),$$
		where $r'=\frac{3}{4}\cdot2^{-n_0}$, we obtain with the help of Lemma \ref{v-doub} that
		\begin{align*}
			\fint_{B(x_0,r')}u^\delta d\mu\le&\ \frac{\mu\left(B(q',r'+R(x_0,q'))\right)}{\mu\left(B(x_0,r')\right)}\fint_{B(q',r'+R(x_0,q'))}u^\delta d\mu\\
			\lesssim&\ \frac{\mu\left(B(q',r'+R(x_0,q'))\right)}{\mu\left(B(q',r'-R(x_0,q'))\right)}\inf\limits_{B(q',r'+R(x_0,q'))}u^\delta\\
			\le&\ \frac{\mu\left(B(q',2^{-n_0})\right)}{\mu\left(B(q',2^{-n_0-2})\right)}\inf\limits_{B(q',r'+R(x_0,q'))}u^\delta\lesssim\inf\limits_{B(x_0,r')}u^\delta.
		\end{align*}
		\item there exist $\omega\in V_{n_0}$ and $q'\in F_\omega(V_2)$ such that
		$$R(x_0\,q')\le 3\cdot 2^{-3-n_0},\quad d_R(x_0,F_\omega(V_0))>2^{-2-n_0},\quad d_R(q',F_\omega(V_0))\ge 2^{-1-n_0}.$$
		Then with $r'=2^{-3-n_0}$, since $B(q',r'+R(x_0,q'))\subset K_\omega$, and $B(x_0,r')$ contains some whole $K_{\omega i_1i_2i_3i_4i_5}$, we have
		\begin{align*}
			\fint_{B(x_0,r')}u^\delta d\mu\le&\ \frac{\mu\left(B(q',r'+R(x_0,q'))\right)}{\mu\left(B(x_0,r')\right)}\fint_{B(q',r'+R(x_0,q'))}u^\delta d\mu\\
			\lesssim&\ \frac{\mu\left(B(q',r'+R(x_0,q'))\right)}{\mu\left(B(x_0,r')\right)}\inf\limits_{B(q',r'+R(x_0,q'))}u^\delta\\
			\le&\ (w_0\wedge w_2)^{-5}\inf\limits_{B(x_0,r')}u^\delta.
		\end{align*}
	\end{itemize}
	Combining (\ref{n-0}), we see $(\mathrm{wEH}_\delta)$ holds completely.
\end{proof}

\end{document}